\documentclass[reqno]{amsart}

\usepackage{amsmath,amssymb,color,amsthm}
\usepackage{amsfonts, amscd, epsfig, amsmath, amssymb,enumerate}
\usepackage{graphicx}
\usepackage{graphics}
\usepackage{color}
\usepackage{mathrsfs}
\usepackage{todonotes}
\usepackage{soul}

\usepackage[usenames,dvipsnames]{pstricks}
\usepackage{pstricks-add}
\usepackage{epsfig}
\usepackage{pst-grad} 
\usepackage{pst-plot} 
\usepackage[space]{grffile} 
\usepackage{etoolbox} 
\usepackage[margin=3cm]{geometry}
\makeatletter 
\patchcmd\Gread@eps{\@inputcheck#1 }{\@inputcheck"#1"\relax}{}{}
\makeatother

\newtheorem{theorem}{Theorem}[section]
\newtheorem{lemma}[theorem]{Lemma}

\newtheorem{remark}[theorem]{Remark}

\usepackage{tikz}
\usetikzlibrary{backgrounds}
\usetikzlibrary{patterns,fadings}
\usetikzlibrary{arrows,decorations.pathmorphing}
\usetikzlibrary{decorations}
\usetikzlibrary{calc}
\usetikzlibrary{shapes.misc}

\usepackage{setspace}

\definecolor{light-gray}{gray}{0.95}
\usepackage{float}
\usepackage[colorlinks=true,linkcolor=blue,citecolor=magenta]{hyperref}
\def\centerarc[#1](#2)(#3:#4:#5){\draw[#1] ($(#2)+({#5*cos(#3)},{#5*sin(#3)})$) arc (#3:#4:#5);}

\newcommand{\vertiii}[1]{{\left\vert\kern-0.25ex\left\vert\kern-0.25ex\left\vert #1 
		\right\vert\kern-0.25ex\right\vert\kern-0.25ex\right\vert}}

\numberwithin{equation}{section}
\numberwithin{figure}{section}

\newcommand{\bb}[1]{{\mathbb #1}}

\newcommand{\<}{\big\langle}
\renewcommand{\>}{\big\rangle}

\renewcommand{\epsilon}{\varepsilon}

\newcommand{\R}{\mathbb R}
\newcommand{\Z}{\mathbb Z}

\renewcommand{\P}{\mathbb P}
\newcommand{\T}{\mathbb T}
\newcommand{\E}{\mathbb E}

\newcommand{\gen}{\mathcal{L}}

\newcommand{\bP}{\mathbb{P}}
\newcommand{\bE}{\mathbb{E}}

\allowdisplaybreaks 

\title[Moderate Deviations for the Current and tagged particle]{Moderate Deviations for the current and Tagged Particle in Symmetric Simple Exclusion Processes}

\author{Xiaofeng Xue}
\address{School of Science, Beijing Jiaotong University, Beijing 100044, China.}
\email{xfxue@bjtu.edu.cn}

\author{Linjie Zhao}
\address{Inria Lille-Nord Europe, Lille 59000, France.}
\email{linjie.zhao@inria.fr}

\begin{document}

	\begin{abstract}
		We prove moderate deviation principles for the tagged particle position and current in one dimensional symmetric simple exclusion processes.  There is at most one particle per site. A particle jumps to one of its two neighbors at rate $1/2$, and the jump is suppressed if there is already one at the target site. We distinguish one particular particle which is called the tagged particle. We first establish a variational formula for the moderate deviation rate functions of the tagged particle positions based on moderate deviation principles from hydrodynamic limit proved by Gao and Quastel \cite{gao2003moderate}. Then we construct a minimizer of the variational formula and obtain explicit expressions for the moderate deviation rate functions.
	\end{abstract}
	
		\maketitle
		
	\section{Introduction}
	
	It has been a long standing problem to investigate the behavior of a \emph{tagged} particle interacting with others, which is closely related to the problem of establishing a rigorous physical basis of Brownian motion.  The main difficulty lies in the fact that the tagged particle itself is in general not Markovian. One of  the simplest interactions between particles is the \emph{exclusion} rule, which means that each site cannot be occupied by more than two particles.  The exclusion process was first introduced by Spitzer\;\cite{spitzer70}, and since then it  has become one of the most popular models in  statistical physics due to its simple structure but rich behaviors, \emph{cf.}\,\cite{liggettips,liggett99}.  The dynamics is as follows: the particle at site $x \in \Z^d,\;d \geq 1,$ waits for an exponential time of parameter one, and then choose a site $y \in \Z^d$ with probability $p(y-x)$, where $p(\cdot)$ is a probability measure on $\Z^d$. If site $y$ is vacant, then the particle at site $x$ jumps to site $y$; otherwise the jump is suppressed and the particle at site $x$ waits for a new exponential time.
	
	Initially, let us put one particle at the origin, and  independently put one at every other site with probability $\rho \in (0,1)$.  We distinguish the particle initially at the origin, and call it the \emph{tagged} particle. In the one dimensional symmetric nearest neighbor case,   Arratia \cite{arratia1983motion} showed that the tagged particle is \emph{sub-diffusive}. More precisely, if we denote by $X_t$ the position of the tagged particle at time $t$, then $ X_t / t^{1/4} $ converges in distribution to the normal distribution with variance $\sqrt{2 / \pi} (1-\rho) / \rho$.  Arratia's proof investigates negative correlation inequalities, and we refer the readers to \cite{DeMasiFerrari02} for a different proof of the above central limit theorem (CLT) based on the relation between current and tagged particle positions. The above CLT was then extended to an invariance principle with respect to the fractional Brownian motion with hurst index $1/4$ in \cite{peligrad2008fractional}. Recently, Sethuraman and Varadhan \cite{sethuraman2013large} established large deviation principles for the position of the tagged particle.
	
	Inspired by the work in \cite{sethuraman2013large}, in this article we consider the moderate deviations of the tagged particle positions, which is the first result in this direction to the best of our knowledge. The proof is based on the following observation:  denote by $J_{x,x+1} (t)$ the net particle current across the bond $(x,x+1)$ up to time $t$, then for $a > 0$,
	\[\{X_t > a\} = \Big\{J_{-1,0} (t) \geq \sum_{x=0}^{a} \eta_t (x) \Big\},\]
	and similarly for $a < 0$,
	\[\{X_t < a\} = \Big\{J_{-1,0} (t) < -  \sum_{x=a}^{-1} \eta_t (x) \Big\}.\]
	Above, $\eta_t (x) = 1$ means there is a particle at site $x$ at time $t$, and $\eta_t (x) = 0$ otherwise. Moreover, formally we have
	\[J_{x,x+1} (t) = \sum_{y >x} \{\eta_t (y) - \eta_0 (y)\}.\]
	Then, following the moderate deviation principles from hydrodynamic limit of symmetric exclusion processes already proved in \cite{gao2003moderate}, the moderate deviation rate function for the current and the tagged particle positions should be given by a variation formula, which is the first main result Theorem \ref{thm:main_1} of the article, and is reminiscent of the contraction principle. We remark that the above strategy has also been used in many other contexts \cite{gonccalves2008central,jara2006nonequilibrium,jara2008quenched,sethuraman2013large}.
	
	The second main contribution of the article is to characterize the moderate deviation rate function explicitly.  The rate functions are quadratic as stated in Theorem \ref{thm:main_2}. To this  end, we construct a minimizer of the variational formula, which permits us to calculate the rate function.
	
	\medspace
	
	{\it Related work}. When the process starts from  stationary measures, Saada \cite{saada1987tagged} proved law of large numbers for the tagged particle positions by proving the ergodicity of the \emph{environment process} as seen from the tagged particle. Then it is a natural question to consider stationary fluctuation.  In the remaining cases except the one dimensional nearest neighbor case, the tagged particle position has been proved to be  diffusive as expected. In the seminal paper \cite{kipnis1986central}, Kipnis and Varadhan have established  a very general result for functional central limit theorems of additive functionals of reversible ergodic Markov chains, based on which they proved invariance principles for the tagged particle positions in symmetric  exclusion processes in all dimensions except the one-dimensional nearest neighbor case. The above approach has been extended to other contexts since then, by Varadhan \cite{varadhan1995self} in the asymmetric mean-zero case , and by Sethuraman, Varadhan and Yau \cite{sethuraman2000diffusive} when $p(\cdot)$ has a drift in dimension $d \geq 3$. We refer the readers to the monograph \cite{komorowski2012fluctuations} for a comprehensive study of the above method. The one dimensional nearest neighbor asymmetric case  was solved by Kipnis using the mapping between exclusion and zero range processes in \cite{Kipnis86}.  In dimensions $d \leq 2$ when the underlying random walk has a drift except the one dimensional nearest-neighbor case, we only know the tagged particle is diffusive as shown by Sethuraman \cite{sethuraman2006diffusive}, and a full CLT or invariance principle remains open.
	Non-equilibrium behaviors of the tagged particle positions have also been investigated in several cases. Law of large numbers was proved in \cite{rezakhanlou1994evolution}. For non-equilibrium fluctuations, \emph{cf.}\;\cite{jara2006nonequilibrium,jara2008quenched} for the one dimensional nearest neighbor case with/without bond disorder, and \cite{jara2009nonequilibriumlong}  for symmetric exclusion process with long jumps. 
	
	The motion of the tagged particle has also been used to check the validity of the Einstein relation \cite{landim1998driven,loulakis2002einstein,loulakis2005mobility}. Very recently, a heat kernel bound for the tagged particle was proved in \cite{giunti2019heat} for symmetric exclusion. The behavior of the particle has also been investigated in other interacting particle processes, such as in zero range processes \cite{jara2009nonequilibrium,jara2013nonequilibrium,sethuraman2007diffusivity} and in stirring-exclusion processes \cite{chen2013limit}.  Besides on the integer lattice $\Z^d$, the tagged particle was also considered on  regular trees \cite{chen2019limit} and on Galton--Watson trees \cite{gantert2020speed}.

	\medspace
	
	The rest of the paper is organized as follows. In  Section \ref{sec:results}, we first recall moderate deviations from hydrodynamic limit in the symmetric exclusion process in Theorem \ref{thm:mdphl}. As stated above, moderate deviation principles for the tagged particles positions and currents are presented in Theorem \ref{thm:main_1} in terms of a variational formula, which is calculated explicitly in Theorem \ref{thm:main_2}. We first prove Theorem \ref{thm:main_2} in Section \ref{sec:proofmain2}, and postpone the proof of Theorem \ref{thm:main_1} to  Section \ref{sec:proofmain1} since the proof  utilizes the properties of rate functions.

	\section{Notation and Results}\label{sec:results}
	
	The one dimensional symmetric simple exclusion process (SSEP) is a Markov process on $\Omega=\{0,1\}^{\Z}$ with infinitesimal generator acting on local functions as
	\begin{equation}\label{generator}
		\gen f (\eta) = \frac{1}{2} \sum_{x \in \Z}  \big\{f(\eta^{x,x+1})  -  f(\eta) \big\}.
	\end{equation}
	Above, a function $f: \Omega \rightarrow \R$ is local if it depends on $\eta$ only through a finite number of sites, and  $\eta^{x,y}$ is the configuration obtained from $\eta$ by exchanging the values of $\eta(x)$ and $\eta(y)$, that is,
	\[\eta^{x,y} (z) = \begin{cases}
		\eta (z), \quad z \neq x,y,\\
		\eta (y), \quad z = x,\\
		\eta (x),\quad z = y.
	\end{cases}\]
	Denote by $\{\eta_t\}_{t \geq 0}$  the Markov process with generator $\gen$. The particles in the SSEP are indistinguishable. We distinguish one particular particle, put it initially at the origin, and call it the \emph{tagged} particle. Denote by $X(t)$ the position of the tagged particle at time $t$. Obviously, $X(0)=0$.
	
	Let $\nu_\rho$, $\rho \in (0,1)$, be the product measure on $\Omega$ with marginals given by
	\[\nu_\rho \{\eta: \eta (x) = 1\} = \rho, \quad \forall x \in \Z.\]
	It is well known that $\nu_\rho$ is reversible and ergodic for the SSEP, \emph{cf.}~\cite{liggettips} for example. Since we are interested in the motion of the tagged particle, we shall start the process from the measure $\nu_\rho^*$ obtained from  $\nu_\rho$ conditioned on having a particle at the origin, i.e.,
	\[\nu_\rho^* (\cdot) = \nu_\rho (\cdot | \eta(0) = 1).\]

	We shall also investigate the behavior of the current in the SSEP. More precisely, for $x \in \bb{Z}$, let $J_{x,x+1} (t)$ be the net number of particles across the bond $(x,x+1)$ up to time $t$, i.e., the number of particles jumping from $x$ to $x+1$ minus the number of particles jumping from $x+1$ to $x$ during the time interval $[0,t]$.
	
	Throughout the paper, we use $\bP$ to denote the measure on $D([0,\infty),\Omega)$ associated to the exclusion process  $\eta_t$ and initial measure $\nu_\rho^*$, and by $\E$ the corresponding expectation.
	
	Recall a function $f: \R \rightarrow \R$ is a Schwartz function if it is smooth and $\sup_{u \in \R} |u^m f^{(n)} (u)| < \infty$ for all\footnote{Here, $\Z_+ = \{0,1,2,\ldots\}.$} $m, n \in \Z_+$.  Let $\mathcal{S} (\bb{R})$ be the space of all Schwartz functions, and  $\mathcal{S}^\prime(\bb{R})$ be the dual of $\mathcal{S} (\bb{R})$. For $m,n \in \Z_+$ and $A,B \subset \R$, we use $C_c^{m,n} (A \times B)$ to denote the space of functions which are $m$-th (resp. $n$-th) continuously differentiable in the first (resp. second) variable and have compact support in $A \times B$. 
	
	\subsection{Moderate deviations from hydrodynamic limit.} In this subsection, we recall moderate deviation principles for the empirical measure of the process already proved by Gao and Quastel\,\cite{gao2003moderate}. For each $N \geq 1$, define the centered empirical measure $\mu^N_t$ as 
	\[\mu^N_t (u) = \frac{1}{a_N} \sum_{x \in \Z} \big( \eta_{tN^2} (x) - \rho \big) \delta_{x/N} (u),\quad u \in \R,\]
	where $\delta_v (u)$ is the Delta function, and $\{a_N\}_{N \geq 1}$ is any positive sequence such that
	\[\lim_{N \rightarrow \infty} \frac{a_N}{N} = \frac{\sqrt{N}}{a_N} = 0.\]
	We regard $\mu^N_t$ as an random element in the space $\mathcal{S}^\prime(\bb{R})$ of tempered distributions. Fix a time horizon $T > 0$. For $G \in C^{1,2}_c ( [0,T]  \times \R)$ and $\mu \in D([0,T],\mathcal{S}^\prime (\bb{R}))$,  let
	\[l(\mu, G)=\left\langle\mu_{T}, G_T\right\rangle-\left\langle\mu_{0}, G_0 \right\rangle-\int_{0}^{T}\langle\mu_{s},\left(\partial_s+(1/2) \partial_u^2\right) G_s \rangle d s.\]
	Above, $G_s(u) = G(s,u)$. The rate function $\mathcal{Q} (\mu) := \mathcal{Q}_{dyn} (\mu)+ \mathcal{Q}_0 (\mu_0)$ is given by
	\begin{equation*}
		\begin{aligned} \mathcal{Q}_{dyn}(\mu) &=\sup _{G \in C^{1,2}_c ( [0,T]  \times \R)}\left\{l(\mu, G)-\frac{\rho(1-\rho)}{2}\int_0^T  \int_\R (\partial_u G)^2 (s,u)\,du\,ds\right\} \\ \mathcal{Q}_{0}\left(\mu_{0}\right) &=\sup _{\phi \in C_c \left(\R\right)}\left\{\left\langle\mu_{0}, \phi\right\rangle-\frac{\rho(1-\rho)}{2} \int_{\R} \phi^2 (u) d u\right\} \end{aligned}
	\end{equation*}
	
	\medspace
	
	\begin{theorem}[Moderate deviations from hydrodynamic limit \cite{gao2003moderate}]\label{thm:mdphl}
		The  process $\{\mu^N_t\}_{0 \leq t \leq T}$ satisfies moderate deviation principle with decay rate $a_N^2 / N$ and with rate function $\mathcal{Q} (\mu)$, more precisely, for any open set $O \in D([0,T], \mathcal{S}^\prime (\R))$,
		\[\liminf_{N \rightarrow \infty} \frac{N}{a_N^2} \P \big( \mu^N_\cdot \in O \big) \geq - \inf_{\mu \in O} \mathcal{Q} (\mu),\]
		and for any closed set $C \in D([0,T], \mathcal{S}^\prime (\R))$,
		\[\liminf_{N \rightarrow \infty} \frac{N}{a_N^2} \P \big( \mu^N_\cdot \in C \big) \leq - \inf_{\mu \in C} \mathcal{Q} (\mu).\]
	\end{theorem}

	\begin{remark}\label{rmk:initialmeasure}
		Although Gao and Quastel \cite{gao2003moderate} proved moderate deviations for  the symmetric exclusion process on the finite ring $\T_N:= \bb{Z} / N \bb{Z}$ starting from the measure $\nu_\rho$, their results could be extended directly to the process evolving on the infinite integer lattice $\Z$ with initial measure $\nu_\rho^*$ by solving some topological issues and by using a simple coupling argument.
		
		The coupling is defined as follows: for two processes $\eta$ and $\xi$, we let them jump as possible as together. This is called basic coupling in the literature, \emph{cf.}\,\cite{liggettips} for example.
	\end{remark}

	\begin{remark}
		Gao and Quastel \cite{gao2003moderate} considered the symmetric exclusion processes in all dimensions. We only state the result for the one dimensional case since it is sufficient for our purpose. 
	\end{remark}

	\subsection{Moderate deviations for the current and tagged particle.}  In this subsection, we state the main results of the article, i.e., moderate deviation principles for the current and tagged particle positions in the symmetric simple exclusion process. 
	
	Since the number of particles is locally conserved, for $x \in \Z$,
	\begin{equation}\label{current_density}
		\eta_{TN^2} (x) - \eta_{0} (x) = J_{x-1,x} (TN^2) - J_{x,x+1} (TN^2).
	\end{equation}
	Summing over $x$ from $0$ to $+\infty$ and divided by $a_N$ on both hand sides, formally we have
	\begin{equation}\label{cd}
		\frac{1}{a_N}J_{-1,0} (TN^2) = \frac{1}{a_N} \sum_{x \geq 0} \big\{\big(\eta_{tN^2} (x) - \rho\big) - \big(\eta_{0} (x) - \rho \big) \big\}.
	\end{equation}
	Note that the infinite sum  on the right-hand side may \emph{not} be well defined.  By Theorem \ref{thm:mdphl} and the contraction principle, $a_N^{-1} J_{-1,0} (TN^2)$ should satisfy the moderate deviation principle with rate function  $\bb{J} (\alpha)$  given by
	\begin{equation}\label{jalpha}
		\bb{J} (\alpha) = \inf \{ \mathcal{Q} (\mu): \mu_T (\chi_{[0,\infty)}) - \mu_0 (\chi_{[0,\infty)}) = \alpha\},\quad \alpha \in \R.
	\end{equation}
	Similarly, for the position of the tagged particle, observe that if  $J_{-1,0}(TN^2)>0$, then
	\begin{equation}\label{jpositive}
		\frac{1}{a_N}J_{-1,0}(TN^2)=\frac{1}{a_N}  \sum_{x=0}^{X(TN^2)-1}\eta_{TN^2}(x) = \frac{1}{a_N}  \sum_{x=0}^{X(TN^2)-1} \big(\eta_{TN^2}(x) - \rho\big)  + \frac{1}{a_N} \rho  X(TN^2),
	\end{equation}
	and if $J_{-1,0}(TN^2) < 0$ , then
	\begin{equation}\label{jnegative}
		\frac{1}{a_N} J_{-1,0}(TN^2)=- \frac{1}{a_N} \sum_{x=X(TN^2)}^{-1}\eta_{TN^2}(x) = - \frac{1}{a_N} \sum_{x=X(TN^2)}^{-1} \big(\eta_{TN^2}(x) - \rho\big)  + \frac{1}{a_N} \rho   X(TN^2).
	\end{equation}
	By large deviation estimates, the sums on the right-hand sides of \eqref{jpositive} and \eqref{jnegative} are both negligible in the limit. Using the contraction principle again, the rate function for the tagged particle positions should be given by
	\begin{equation}\label{ialpha}
		\bb{I} (\alpha) = \inf \{ \mathcal{Q} (\mu): \mu_T (\chi_{[0,\infty)}) - \mu_0 (\chi_{[0,\infty)}) = \rho \alpha \}, \quad \alpha \in \R.
	\end{equation}
	The first main result of the article validates the above formal arguments.
	
	\begin{theorem}\label{thm:main_1}
		The sequence of the currents $\{ J_{-1,0} (TN^2) / a_N\}_{N \geq 1}$, respectively of the tagged particle positions $\{X(TN^2)/ a_N\}_{N \geq 1}$,  satisfies moderate deviation principles with decay rate $a_N^2 / N$ and with rate function $\bb{J} (\alpha)$, respectively with $\bb{I} (\alpha)$. To be precise, for any open set $O \in \R$,
		\begin{align*}
			\liminf_{N \rightarrow \infty} \frac{N}{a_N^2} \log \P \Big( \frac{1}{a_N} J_{-1,0} (TN^2) \in O \Big) \geq - \inf_{\alpha \in O} \bb{J} (\alpha),\\
			\liminf_{N \rightarrow \infty} \frac{N}{a_N^2}\log \P \Big( \frac{1}{a_N} X(TN^2) \in O \Big) \geq - \inf_{\alpha \in O} \bb{I} (\alpha),
		\end{align*}
		and for any closed set $C \in \R$,
		\begin{align*}
			\limsup_{N \rightarrow \infty} \frac{N}{a_N^2} \log \P \Big( \frac{1}{a_N} J_{-1,0} (TN^2) \in C \Big) \leq - \inf_{\alpha \in C} \bb{J} (\alpha),\\
			\limsup_{N \rightarrow \infty} \frac{N}{a_N^2} \log \P \Big( \frac{1}{a_N} X(TN^2) \in C\Big) \leq - \inf_{\alpha \in C} \bb{I} (\alpha).
		\end{align*}
	\end{theorem}
	
	\medspace

	Next, we state an explicit formula for the rate functions $\bb{J} (\alpha)$ and $\bb{I} (\alpha)$. It was proved in \cite{arratia1983motion,DeMasiFerrari02} the following central limit theorems for the current and tagged particle, 
	\[\lim_{t \rightarrow \infty} \frac{J_{-1,0} (t)}{t^{1/4}} = \mathscr{N} (0,\sigma_J^2), \quad 
	\lim_{t \rightarrow \infty} \frac{X (t)}{t^{1/4}} = \mathscr{N} (0,\sigma_X^2)\]
	in distribution, where $\mathscr{N}  (0,\sigma^2)$ is the normal distribution with mean zero and variance $\sigma^2$, and
	\[\sigma_X^2 = \sqrt{\frac{2}{\pi}} \frac{1-\rho}{\rho}, \quad \sigma_J^2 = \sqrt{\frac{2}{\pi}} \rho (1-\rho).\]
	We also remark that the corresponding invariance principle was proved in \cite{peligrad2008fractional} with respect to the  fractional Brownian motion with parameter $1/4$. Then, a heuristic argument shows that the rate functions for the current and tagged particle positions should be $\alpha^2 / (2 \sigma_J^2 \sqrt{T})$ and $\alpha^2 / (2 \sigma_X^2 \sqrt{T})$. Indeed, for any $\alpha \in \R$, formally,
	\[	\bb{P} \Big( \frac{1}{a_N} J_{-1,0} (TN^2)  = \alpha \Big) = \bb{P} \Big( \frac{J_{-1,0} (TN^2)}{N^{1/2}T^{1/4}}   = \frac{\alpha a_N}{N^{1/2}T^{1/4}} \Big) \approx \frac{1}{\sqrt{2 \pi \sigma_J^2}} \exp \Big\{ - \frac{1}{2 \sigma_J^2} \frac{\alpha^2 a_N^2}{N \sqrt{T}}\Big\}.\]
	Therefore,
	\[\frac{N}{a_N^2} \log \bb{P} \Big( \frac{1}{a_N} J_{-1,0} (TN^2) = \alpha \Big) \approx - \frac{\alpha^2}{2 \sigma_J^2 \sqrt{T}}. \]
	The same argument is also true for the tagged particle positions. The following result verifies that this is indeed the case, which is the second main result of the article.
	
	\begin{theorem}\label{thm:main_2}
		For the rate functions, we have
		\[\bb{J} (\alpha) = \frac{\sqrt{2\pi} \alpha^2}{4 \rho (1-\rho)\sqrt{T}}, \quad
		\bb{I} (\alpha) = \frac{\sqrt{2\pi} \rho \alpha^2}{4 (1-\rho)\sqrt{T}}.\]
	\end{theorem}
	
	\medspace
	
	\section{Proof of Theorem \ref{thm:main_2}}\label{sec:proofmain2}
	
	In this section, we prove Theorem \ref{thm:main_2}. Comparing \eqref{ialpha} with  \eqref{jalpha}, we have $\bb{I} (\alpha) = \bb{J} (\rho \alpha)$. Therefore, we only need to prove Theorem \ref{thm:main_2} for the rate function $\bb{J} (\alpha)$.
	
	\begin{proof}[Proof of Theorem \ref{thm:main_2}]
		By the definition of $\bb{J} (\alpha)$ in \eqref{jalpha}, we only need to consider the family of  tempered distributions $\mu$ such that $\mathcal{Q} (\mu) < \infty$.  By Riesz's representation theorem,  if $\mathcal{Q}_0 (\mu_0) < \infty$, then there exists $\psi \in L^2 (\R)$ such that
		\begin{equation}\label{mu0}
			\<\mu_0,\phi\> = \int_{\R} \psi(u) \phi (u) du, \quad \mathcal{Q}_0 (\mu_0) = \frac{||\psi||_{L^2 (\R)}^2}{2 \rho (1-\rho).}
		\end{equation}
		For $H,G \in C^{1,2}_c ([0,T] \times \R)$, define
		\[[H,G] = \int_{0}^T \int_{\R} \partial_{u} H(t,u) \partial_{u}G(t,u)\,du\,dt.\]
		We say $H \sim G$ if $[H-G,H-G] = 0$. Let $\mathcal{H}^1$ be the Hilbert space obtained as the completion of  $C^{1,2}_c ([0,T] \times \R)) / \sim$ with respect to the inner product $[\cdot,\cdot]$. By \cite[Lemma 5.1]{gao2003moderate} and Eq.\,\eqref{mu0}, if $\mathcal{Q} (\mu) < \infty$, then  there exist some $H \in \mathcal{H}^1$ and $\psi \in L^2 (\R)$ such that $\mu$ is the unique solution of 
		\begin{equation}\label{mu}
			\begin{cases}
				\partial_t \mu (t,u) = (1/2) \partial_u^2 \mu (t,u) - \rho (1-\rho)  \partial_u^2 H (t,u)\\
				\mu_0 (u) = \psi (u).
			\end{cases}
		\end{equation}
		Moreover, 
		\begin{equation}\label{q_dyn}
			\mathcal{Q}_{dyn}(\mu) = \frac{\rho(1-\rho)}{2} \int_0^T \int_{\R} (\partial_u H )^2 (t,u) \,du\, dt.
		\end{equation}
		To sum up, for any $\mu$ such that $\mathcal{Q} (\mu) < \infty$, we could find $H_n \in C^{1,2}_c ([0,T] \times \R))$ and $\psi_n \in C^2_c (\R)$ such that the corresponding $\mu_n$ satisfies $\mathcal{Q} (\mu_n) \rightarrow \mathcal{Q} (\mu)$ as $n \rightarrow \infty$.  Denote 
		\[\mathcal{A} = \big\{ \mu: \text{there exist}\;H \in C^{1,2}_c ([0,T] \times \R)) \;\text{and}\; \psi \in C^2_c (\R) \;\text{such that $\mu$ satisfies Eq.\,\eqref{mu}} \big\}.\]
		Then 
		\begin{equation}\label{jalpha1}
			\bb{J} (\alpha) = \inf \{ \mathcal{Q} (\mu): \mu \in \mathcal{A},\;\mu_T (\chi_{[0,\infty)}) - \mu_0 (\chi_{[0,\infty)}) = \alpha\},\quad \alpha \in \R.
		\end{equation}
		
		\medspace
		
		Now take $\mu \in \mathcal{A}$ such that $\mu_T (\chi_{[0,\infty)}) - \mu_0 (\chi_{[0,\infty)}) = \alpha$. Note that in this case $\mu$ actually satisfies $\mu \in C_c^{1,2} ([0,T] \times \R)$.  Denote
		\[J = -\frac{1}{2} \partial_u \mu + \rho (1-\rho) \partial_u H.\] 
		Then by \eqref{mu},
		\[\partial_t \mu + \partial_u J =0.\]
		Replacing $\partial_{u} H$ in \eqref{q_dyn} with $[\rho(1-\rho)]^{-1} (J + (1/2) \partial_{u} \mu)$, we may rewrite $\mathcal{Q}_{dyn}$ as
		\begin{equation}\label{qdyn}
			\mathcal{Q}_{dyn}(\mu) = \frac{1}{\rho(1-\rho)} \int_0^T \int_{\R} \Big\{\frac{1}{2} J^2 (t,u) + \frac{1}{8}( \partial_u \mu)^2 (t,u) + \frac{1}{2} J (t,u) \partial_u \mu (t,u)\Big\} \,du \,dt.
		\end{equation}
		Using the integration by parts formula,
		\begin{multline*}
			\int_0^T \int_{\R}\frac{1}{2} J (t,u) \partial_u \mu (t,u)  \,du \,dt = - \int_0^T \int_{\R}\frac{1}{2}  \mu (t,u) \partial_u J (t,u) \,du \,dt  \\
			= \int_0^T \int_{\R}\frac{1}{2}  \mu (t,u) \partial_t \mu (t,u) \,du \,dt = \frac{1}{4} \int_{\R} \Big\{\mu^2 (T,u) - \mu_0^2 (u)\Big\} du. 
		\end{multline*}
		Therefore,
		\[ \rho (1-\rho) \mathcal{Q}_{dyn}(\mu)= \int_0^T \int_{\R} \Big\{\frac{1}{2} J^2 (t,u) + \frac{1}{8}( \partial_u \mu)^2 (t,u) \Big\} \,du \,dt + \frac{1}{4} \int_{\R} \Big\{\mu^2 (T,u) - \mu^2_0 (u)\Big\} du. \]
		By Eq.\,\eqref{mu0},
		\[\mathcal{Q}_0 (\mu_0) = \frac{1}{2 \rho (1-\rho)} \int_{\R} \mu_0^2 (u) \,du. \]
		Therefore, 
		\begin{equation}\label{q}
			\rho (1-\rho) \mathcal{Q} (\mu)= \int_0^T \int_{\R} \Big\{\frac{1}{2} J^2 (t,u) + \frac{1}{8}( \partial_u \mu)^2 (t,u) \Big\} \,du \,dt + \frac{1}{4} \int_{\R} \Big\{\mu^2 (T,u) + \mu_0^2 (u)\Big\} du.
		\end{equation}
		
		\medspace
		
		Define $K: \R_+ \times \R \rightarrow \R$ as
		\[K (t,u) = \int_0^t J (s,u) ds.\]
		Direct calculations yield that
		\begin{align}
			&\partial_t K (t,u) = J(t,u),\notag\\
			&\partial_u K (t,u) = \int_{0}^t \partial_{u} J(s,u) du = - \int_0^t \partial_s \mu(s,u) ds = \mu_0 (u) - \mu (t,u),\label{eqn:kmu}\\
			& \partial_u^2 K (t,u) = \partial_u \mu_0 (u) -  \partial_u \mu (t,u).\notag
		\end{align}
		Replacing $J(t,u), \partial_{u} \mu(t,u)$ and $\mu(T,u)$ in Eq.\,\eqref{q} by functional of $K(t,u)$ and $\mu_0 (u)$, we may rewrite the rate function $\mathcal{Q} (\mu)$ as 
		\begin{multline}\label{qmu1}
			\rho (1-\rho) \mathcal{Q} (\mu) = \int_0^T \int_{\R} \Big\{\frac{1}{2} (\partial_t K)^2 (t,u) + \frac{1}{8}( \partial_{u}^2 K)^2 (t,u) \Big\} \,du \,dt + \int_{\R} \frac{1}{4} (\partial_{u} K)^2 (T,u)\,du\\
			+ \int_0^T \int_{\R} \Big\{\frac{1}{8} (\partial_u \mu_0)^2 (u) - \frac{1}{4} \partial_u \mu_0 (u) \partial_{u}^2 K (t,u) \Big\} \,du \,dt + \int_{\R} \frac{1}{2} \mu_0^2 (u) - \frac{1}{2} \mu_0 (u) \partial_u K (T,u)\,du.
		\end{multline}
		Moreover, the constraint $\mu_T (\chi_{[0,\infty)}) - \mu_0 (\chi_{[0,\infty)}) = \alpha$ reduces to that
		\[K(T,0) = \int_0^T J(s,0) ds = -  \int_0^T \int_0^\infty \partial_u J(s,u) du ds = \int_0^\infty \mu(T,u) - \mu_0(u) du = \alpha.\]
		Obviously, $K(0,u) = 0$ for all $u \in \R$.  Observe that the expression of Eq.\,\eqref{qmu1} is similar to that of $\mathcal{K}$ in \cite[Page 1493]{sethuraman2013large}, except that we need to consider the effect of the initial conditions. To obtain the infimum of \eqref{qmu1} over the above constraints, we first state a lemma without proof.
		
		\begin{lemma}[{\cite[Proposition 4.4]{sethuraman2013large}}]\label{lem:sv}
			For $M \in C^{1,2} ([0,1] \times \R)$, define
			\[ \mathcal{M}= \frac{1}{4} \int\left|\partial_{u} M(1, u)\right|^{2} d u+\frac{1}{2} \int_{0}^{1} \int\left|\partial_t M (t, u)\right|^{2} d u d t 
			+\frac{1}{8} \int_{0}^{1} \int\left|\partial_{u}^2 M (t, u)\right|^{2} d u d t.\]
			Then 
			\[\inf \{ \mathcal{M}: M(0,u) \equiv 0, M(1,0)=1 \}= \frac{\sqrt{\pi}}{2}.\]
			Moreover, the infimum is obtained  in $M$ such that
			\[\frac{1}{4} \partial_u^4 M (t,u) = \partial_t^2 M (t,u), \quad \partial_u^2 M (1,u) = 0, \quad \partial_t M (1,u) = 0.\]
		\end{lemma}
		
		For $K \in C^{1,2} ([0,T] \times \R)$, denote by $F_T (K)$ the first line on the right-hand side of \eqref{qmu1}, \emph{i.e.}
		\[F_{T} (K) = \int_0^T \int_{\R} \Big\{\frac{1}{2} (\partial_t K)^2 (t,u) + \frac{1}{8}( \partial_{u}^2 K)^2 (t,u) \Big\} \,du \,dt + \int_{\R} \frac{1}{4} (\partial_{u} K)^2 (T,u)\,du.\]
		Let $M (t,u) = \alpha^{-1} K(tT,u\sqrt{T})$. Then we could rewrite $F_T (K)$ as
		\[F_{T} (K) = \frac{\alpha^2}{\sqrt{T}} \mathcal{M},\]
		where $\mathcal{M}$ is given in Lemma \ref{lem:sv}. Moreover, the constraints on $K$ reduces to that $M(0,u) \equiv 0$ and that $M(1,0)=1$.  Therefore, using Lemma \ref{lem:sv},
		\[\inf\{F_{T} (K): K (T,0) = \alpha, \;K(0,u) = 0\} = \frac{\sqrt{\pi} \alpha^2}{2 \sqrt{T}},\]
		and the infimum is attained at the point $K_{\alpha,T}$  such that
		\begin{equation}\label{kalphat}
			\frac{1}{4} \partial_u^4 K_{\alpha,T} (t,u) = \partial_t^2 K_{\alpha,T} (t,u), \quad \partial_{u}^2 K_{\alpha,T} (T,u) \equiv 0,\quad \partial_t K_{\alpha,T} (T,u) \equiv 0.
		\end{equation}

		\medspace
		
		Next, we shall construct the point $(\widetilde{K}_{\alpha,T},\widetilde{\mu}_{\alpha,T} (0,\cdot))$ at which  the infimum of \eqref{qmu1} is attained from the above observations. 
		Denote by $\mu_{\alpha,T}$ the corresponding density associated to $K_{\alpha,T}$ with initial condition $\mu_{\alpha,T} (0,u) = 0$.  More precisely, by \eqref{eqn:kmu}, $\mu_{\alpha,T} (t,u) = - \partial_{u} K_{\alpha,T} (t,u)$. Note that the parameter $\alpha$ could be interpreted as the flux from the left of the origin to the right up to time $T$. The main idea is to construct $\widetilde{\mu}_{\alpha,T}$ associated to $\widetilde{K}_{\alpha,T}$ in such a way that the flux  from the left of the origin to the right is $\alpha/2$ during the time intervals $[0,T/2]$ and $[T/2,T]$, and that $\widetilde{\mu}_{\alpha,T} (T/2,u) = 0$ for all $u$.  
		
		To this aim, define
		\begin{equation}\label{mu_tilde}\widetilde{\mu}_{\alpha,T} (t,u) = \begin{cases}
				-\mu_{\alpha/2,T/2} \big(\tfrac{T}{2}-t,u\big)\quad &\text{if}\quad 0 \leq t \leq T/2,\\
				\mu_{\alpha/2,T/2} \big(t -\tfrac{T}{2},u\big)\quad &\text{if}\quad T/2 \leq t \leq T.
		\end{cases}\end{equation}
		Then $\widetilde{\mu}_{\alpha,T} (0,u) = - \mu_{\alpha/2,T/2} \big(\tfrac{T}{2},u\big)$.  By \eqref{eqn:kmu}, the corresponding $\widetilde{K}_{\alpha,T}$ is given by
		\begin{equation}\label{ktilde}
			\widetilde{K}_{\alpha,T} (t,u) = \begin{cases}
				K_{\alpha/2,T/2} \big(\tfrac{T}{2},u\big)- K_{\alpha/2,T/2} \big(\tfrac{T}{2}-t,u\big)\quad &\text{if}\quad 0 \leq t \leq T/2,\\
				K_{\alpha/2,T/2} \big(\tfrac{T}{2},u\big)+ K_{\alpha/2,T/2} \big(t-\tfrac{T}{2},u\big)\quad &\text{if}\quad T/2 \leq t \leq T.
			\end{cases}
		\end{equation}
		Note also that 
		\[\widetilde{\mu}_{\alpha,T} (0,u) = \partial_u K_{\alpha/2,T/2} \big(\tfrac{T}{2},u\big).\]
		We claim that the infimum of  $\rho (1-\rho) \mathcal{Q} (\mu)$ as in \eqref{qmu1} is attained at the point $(\widetilde{K}_{\alpha,T},\widetilde{\mu}_{\alpha,T} (0,\cdot))$.   To simplify notations, we write $K = K_{\alpha/2,T/2}$ when there is no ambiguity. A direct calculation shows that the infimum  is given by
		\begin{align}
			\int_0^{T/2} \int_\R &\Big\{ \frac{1}{2} \Big(\partial_t K (\tfrac{T}{2} -t,u) \Big)^2 + \frac{1}{8} \Big(\partial_u^2 K (\tfrac{T}{2} ,u) -  \partial_u^2 K (\tfrac{T}{2} - t ,u)\Big)^2 +  \frac{1}{8} \Big(\partial_u^2 K (\tfrac{T}{2} ,u) \Big)^2 \notag\\
			&\qquad -  \frac{1}{4} \partial_u^2 K (\tfrac{T}{2} ,u) \Big(\partial_u^2 K (\tfrac{T}{2} ,u) - \partial_u^2 K (\tfrac{T}{2} -t,u)\Big) \Big\}du\,dt\label{inf1}\\
			+ 	\int_{T/2}^{T} \int_\R &\Big\{ \frac{1}{2} \Big(\partial_t K (t- \tfrac{T}{2} ,u) \Big)^2 + \frac{1}{8} \Big(\partial_u^2 K (\tfrac{T}{2} ,u) +  \partial_u^2 K (t-\tfrac{T}{2} ,u)\Big)^2 +  \frac{1}{8} \Big(\partial_u^2 K (\tfrac{T}{2} ,u) \Big)^2 \notag\\
			&\qquad -  \frac{1}{4} \partial_u^2 K (\tfrac{T}{2} ,u) \Big(\partial_u^2 K (\tfrac{T}{2} ,u) + \partial_u^2 K (t-\tfrac{T}{2} ,u)\Big) \Big\}du\,dt\label{inf2}\\
			+ \int_{\R} \frac{1}{4} &\Big(2 \partial_u K (\tfrac{T}{2},u) \Big)^2 + \frac{1}{2} \Big(\partial_u K (\tfrac{T}{2},u) \Big)^2 - \frac{1}{2} \partial_u K (\tfrac{T}{2},u)  \Big(2 \partial_u K (\tfrac{T}{2},u) \Big) du.\label{inf3}
		\end{align}
		The sum of \eqref{inf1} and \eqref{inf2} is
		\[2 \int_{0}^{T/2} \int_{\R} \frac{1}{2} \Big(\partial_t K (t,u) \Big)^2 + \frac{1}{8} \Big(\partial_u^2 K (t,u) \Big)^2\,du\,dt,\]
		and \eqref{inf3} equals
		\[2 \int_{\R} \frac{1}{4} \Big( \partial_u K (\tfrac{T}{2},u) \Big)^2 \,du.  \]
		Therefore,
		\[\bb{J} (\alpha) = \frac{2}{\rho(1-\rho)} F_{T/2} \big(K_{\alpha/2,T/2}\big) = \frac{\sqrt{2\pi} \alpha^2}{4 \rho (1-\rho)\sqrt{T}}.\]
		
		\medspace
		
		It remains to prove the claim. It is easy to see that $\widetilde{K}_{\alpha,T}$ satisfies the constraints 
		\[\widetilde{K}_{\alpha,T} (0,u) \equiv 0, \quad \widetilde{K}_{\alpha,T} (T,0) = \alpha.\]
		Denote by $G(K,\mu_0)$ the formula on the right-hand side of \eqref{qmu1}. Then we only need to prove that for any $K \in C^{1,2} ([0,T] \times \R)$ such that $K(T,0) = 0$ and $K(0,u) \equiv 0$, and for any $\mu_0 \in C_c^2 (\R)$, we have
		\[G(K+\widetilde{K}_{\alpha,T},\mu_0+\widetilde{\mu}_{\alpha,T} (0,\cdot))   \geq G(\widetilde{K}_{\alpha,T},\widetilde{\mu}_{\alpha,T} (0,\cdot)).\]
		To make notations short, below we write $\widetilde{K} = \widetilde{K}_{\alpha,T}$ and $\widetilde{\mu}_{0} (\cdot) = \widetilde{\mu}_{\alpha,T} (0,\cdot)$. Direct calculations show that
		\begin{align}
			G&(K+\widetilde{K},\mu_0+\widetilde{\mu}_{0} )   - G(\widetilde{K},\widetilde{\mu}_0)\notag\\
			&\geq \int_{0}^T \int_{\R} \partial_t K (t,u) \partial_t \widetilde{K} (t,u) + \frac{1}{8} \big(\partial_u^2 K (t,u)\big)^2 +  \frac{1}{4} \partial_u^2 \widetilde{K} (t,u) \partial_u^2 K (t,u) \,du\, dt\label{d1}\\
			&+ \int_{\R} \frac{1}{4} \big(\partial_u K (T,u)\big)^2 + \frac{1}{2}  \partial_u \widetilde{K} (T,u) \partial_u K (T,u)\,du\label{d2}\\
			&+ \int_{0}^T \int_{\R} \Big\{ \frac{1}{8}  \big(\partial_u \mu_0(u)\big)^2 + \frac{1}{4} \partial_u \widetilde{\mu}_0(u) \partial_u \mu_0(u) \notag\\
			&\qquad - \frac{1}{4} \partial_u \mu_0(u) \partial_u^2 K (t,u) 
			- \frac{1}{4} \partial_u \mu_0(u) \partial_u^2 \widetilde{K} (t,u) - \frac{1}{4} \partial_u \widetilde{\mu}_0(u) \partial_u^2 K (t,u) \Big\} \,du\, dt\label{d3}\\
			&+ \int_{\R} \frac{1}{2} \mu_0(u)^2 +\widetilde{\mu}_0(u)   \mu_0 (u) - \frac{1}{2} \mu_0 (u) \partial_u K (T,u) -  \frac{1}{2} \widetilde{\mu}_0 (u) \partial_u K (T,u) -  \frac{1}{2} \mu_0 (u) \partial_u \widetilde{K} (T,u) du.\label{d4}
		\end{align}
		
		Next we shall deal with the above four terms respectively. To treat the term \eqref{d1}, using the integration by parts formula, we have
		\begin{multline*}
			\int_{0}^T \int_{\R} \partial_t K (t,u) \partial_t \widetilde{K} (t,u) + \frac{1}{4} \partial_u^2 \widetilde{K} (t,u) \partial_u^2 K (t,u) \,du\, dt \\
			= 	\int_{0}^T \int_{\R}  \Big( - \partial_t^2 \widetilde{K} (t,u) + \frac{1}{4} \partial_u^4 \widetilde{K} (t,u) \Big) K (t,u)  \,du\,dt + \int_{\R} K (T,u) \partial_t \widetilde{K} (T,u) du.
		\end{multline*}
		Recall the definition of $\widetilde{K}$ in \eqref{ktilde}. By \eqref{kalphat},
		\begin{equation}\label{p1}
			\partial_t^2 \widetilde{K} (t,u) = \frac{1}{4} \partial_u^4 \widetilde{K} (t,u),\quad   \partial_t \widetilde{K} (T,u) = \partial_t K_{\alpha/2,T/2} (\tfrac{T}{2},u) \equiv 0.
		\end{equation}
		Therefore, the term \eqref{d1} reduces to 
		\[\int_{0}^T \int_{\R}  \frac{1}{8} \big(\partial_u^2 K (t,u)\big)^2 \,du\, dt.\]
		Note also that for all $u \in \R$,
		\begin{equation}\label{p2}
			\partial_u^2 \widetilde{K} (T,u) = 2 \partial_{u}^2 K_{\alpha/2,T/2} (\tfrac{T}{2},u) = 0, \quad \partial_u \widetilde{\mu}_0 (u) =  \partial_{u}^2 K_{\alpha/2,T/2} (\tfrac{T}{2},u) = 0.
		\end{equation}
		The integration by parts formula implies that the integral of the second integrand in \eqref{d2} is zero. This reduces \eqref{d2} to 
		\[\int_{\R} \frac{1}{4} \big(\partial_u K (T,u)\big)^2 \,du.\]
		By \eqref{p2}, the term \eqref{d3} equals
		\[\int_{0}^T \int_{\R} \Big\{ \frac{1}{8}  \big(\partial_u \mu_0(u)\big)^2 
		- \frac{1}{4} \partial_u \mu_0(u) \partial_u^2 K (t,u) 
		- \frac{1}{4} \partial_u \mu_0(u) \partial_u^2 \widetilde{K} (t,u) \Big\} \,du\, dt.\]
		Since for all $u \in \R$,
		\[\int_0^T \partial_u^2 \widetilde{K} (t,u) dt =- \int_{0}^{T/2}\partial_u^2   K_{\alpha/2,T/2} \big(\tfrac{T}{2}-t,u\big)dt+\int_{T/2}^{T}\partial_u^2   K_{\alpha/2,T/2} \big(t-\tfrac{T}{2},u\big)dt = 0,  \]
		we finally write \eqref{d3} as
		\[\int_{0}^T \int_{\R} \Big\{ \frac{1}{8}  \big(\partial_u \mu_0(u)\big)^2 
		- \frac{1}{4} \partial_u \mu_0(u) \partial_u^2 K (t,u) 
		\Big\} \,du\, dt.\]
		For the last term \eqref{d4}, first note that the integral of the fourth integrand is zero by integration by parts and \eqref{p2}. Since $\partial_u \widetilde{K} (T,u) = 2 \tilde{\mu_0} (0)$,  the second and the last integrands cancel out. Therefore, \eqref{d4} is equal to
		\[\int_{\R} \frac{1}{2} \mu_0(u)^2 - \frac{1}{2} \mu_0 (u) \partial_u K (T,u)\,du.\]
		To sum up, we have shown that
		\begin{multline*}
			G (K+\widetilde{K},\mu_0+\widetilde{\mu}_{0} )   - G(\widetilde{K},\widetilde{\mu}_0) \geq \frac{1}{8} \int_{0}^T \int_{\R}  \Big( \partial_u^2 K (t,u)  - \partial_u \mu_0(u) \Big)^2\,du\,dt\\
			+ \int_{\R} \frac{1}{4} \mu_0(u)^2 + \frac{1}{4} \Big( \mu_0(u) -\partial_u K (T,u) \Big)^2 \,du \geq 0.
		\end{multline*}
		This proves the claim and then concludes the proof of the theorem.
	\end{proof}

	\section{Proof of Theorem \ref{thm:main_1}}\label{sec:proofmain1}
	
	In this section, we prove Theorem \ref{thm:main_1}. We prove exponential tightness in Lemma \ref{lem:exptight}. In subsections \ref{subsec:upper} and \ref{subsec:lower}, we prove weak moderate deviation upper and lower bounds respectively.

	In order to make \eqref{cd} rigorous, we need to introduce an approximation function. For $n \geq 1$, let 
	\[G_n (u) = \chi_{\{u>0\}} (1-u/n)^+, \quad u \in \R.\]
	Multiplying by $G_n (x/N)$ on both hand sides of \eqref{current_density}, and summing over $x$, we have
	\begin{equation}\label{current_empiricalM}
		\<\mu^N_T,G_n\> - \<\mu^N_0,G_n\> = - \frac{1}{nNa_N} \sum_{x=0}^{nN-1} J_{x,x+1} (TN^2) + \frac{1}{a_N} J_{-1,0} (TN^2).
	\end{equation}
	The above formula permits us to derive moderate deviations for the current once we show the first term on the right-hand side is super-exponentially small, \emph{cf.}~Lemma \ref{lem:superexponential}. Then we  prove moderate deviations for tagged particle positions based on \eqref{jpositive} and \eqref{jnegative}.
	
	\subsection{A super-exponential estimate.} In this subsection, We  show that the first term on the right hand side of \eqref{current_empiricalM} is super-exponentially small
	by exploiting the exponential martingale associated to the currents as in \cite[Proposition 3.1]{sethuraman2013large} and based on a standard estimation on the largest eigenvalue of the perturbation of the generator.
	
	\begin{lemma}\label{lem:superexponential}
		For any $\delta > 0$,
		\[\limsup_{n \rightarrow \infty} 	\limsup_{N \rightarrow \infty} \frac{N}{a_N^2} \log \P \Big( \big| \frac{1}{nNa_N} \sum_{x=0}^{nN-1} J_{x,x+1} (tN^2)   \big|> \delta \Big) = - \infty.\]
	\end{lemma}
	
	\begin{proof}
		For any $K > 0$, by Markov inequality, we bound the expression in the lemma by
		\begin{equation}\label{est_1}
			-\delta K + \frac{N}{a_N^2} \log \E \Big[ \exp \Big\{\big| \frac{a_NK}{nN^2} \sum_{x=0}^{nN-1} J_{x,x+1} (tN^2)   \big| \Big\}\Big].
		\end{equation}
		Since $e^{|x|} \leq e^x + e^{-x}$ and $\log (a+b) \leq 2 \max \{\log a,\log b\}$ for any $a,b >0$, we could remove the absolute value above. We start with the observation that
		\[\exp \Big\{\frac{2 a_NK}{nN^2} \sum_{x=0}^{nN-1} J_{x,x+1} (tN^2)  -2 \Gamma_{n,K}^N (t;\eta_\cdot)\Big\}\]
		is a mean one martingale, where 
		\begin{multline*}
			\Gamma_{n,K}^N (t;\eta_\cdot) = (1/4)\big(e^{2a_NK/(nN^2)}-1\big) \sum_{x=0}^{nN-1} \int_0^{tN^2} \eta_s(x) (1-\eta_s(x+1)) ds \\
			+ (1/4)\big(e^{-2a_NK/(nN^2)}-1\big) \sum_{x=0}^{nN-1} \int_0^{tN^2} \eta_s(x+1) (1-\eta_s(x)) ds.
		\end{multline*}
		The above martingale comes from the fact that $\{J_{x,x+1}^{\pm} (t)\}_{x\in \Z}$ are mutually independent compound Poisson processes since there are no jumps occurring at the same time, and that the intensity for $J_{x,x+1}^+ (t)$ (resp.\;for $J_{x,x+1}^- (t)$) is $(1/2) \eta_t (x) (1-\eta_t(x+1))$ (resp.\;$(1/2)\eta_t (x+1) (1-\eta_t(x))$).  Above, $J_{x,x+1}^+ (t)$  (resp.\;$J_{x,x+1}^- (t)$) is the number of jumps of particles from $x$ (resp.\;$x+1$) to $x+1$ (resp.\;$x$) up to time $t$, and therefore $J_{x,x+1} = J_{x,x+1}^+ - J_{x,x+1}^-$. 
		
		Using the basic inequality $E[e^X] \leq (E[e^{2(X-Y)}])^{1/2} (E[e^{2Y}])^{1/2}$ for any two random variables $X,Y$, we may bound the second term in \eqref{est_1} by
		\[ \frac{N}{2 a_N^2} \log \E \Big[ e^{2 \Gamma_{n,K}^N (t;\eta_\cdot)} \Big].\]
		We further write $\Gamma_{n,K}^N (t;\eta_\cdot) = \Gamma_{n,K}^{N,1} (t;\eta_\cdot) + \Gamma_{n,K}^{N,2} (t;\eta_\cdot)$, where
		\begin{multline*}
			\Gamma_{n,K}^{N,1} (t;\eta_\cdot) = (1/4)\big(e^{2a_NK/(nN^2)}-1 - 2a_NK/(nN^2)\big) \sum_{x=0}^{nN-1} \int_0^{tN^2} \eta_s(x) (1-\eta_s(x+1)) ds \\
			+ (1/4)\big(e^{-2a_NK/(nN^2)}-1 + 2a_NK/(nN^2)\big) \sum_{x=0}^{nN-1} \int_0^{tN^2} \eta_s(x+1) (1-\eta_s(x)) ds,
		\end{multline*}
		and
		\[\Gamma_{n,K}^{N,2} (t;\eta_\cdot) = \frac{1}{2} \frac{a_N K}{nN^2} \sum_{x=0}^{nN-1} \int_0^{tN^2} \big(\eta_s(x) - \eta_s(x+1)\big) ds.\]
		Then, we have 
		\begin{equation}\label{est_2}
			\frac{N}{2 a_N^2} \log \E \Big[ e^{2 \Gamma_{n,K}^N (t;\eta_\cdot)} \Big] \leq  \frac{N}{4 a_N^2} \log \E \Big[ e^{4 \Gamma_{n,K}^{N,1} (t;\eta_\cdot)} \Big] + \frac{N}{4 a_N^2} \log \E \Big[ e^{4 \Gamma_{n,K}^{N,2} (t;\eta_\cdot)} \Big].
		\end{equation}

		Using the basic inequality $e^x - 1-x \leq (1/2)x^2 e^{|x|}$, there exists a finite constant $C$ so that we may bound $\Gamma_{n,K}^{N,1} (t;\eta_\cdot) $ from above by  $Cta_N^2K^2 / (nN)$. Whence, the first term on the right-hand side in \eqref{est_2} is bounded by $CtK^2/n$, which converges to zero as $n \rightarrow \infty$. By Feynman-Kac formula (\emph{cf.}\;\cite[Lemma A.1.7.2]{klscaling}), we may bound the second term in \eqref{est_2} by
		\[\frac{tN^3}{4a_N^2} \sup_{f:\nu_\rho-\text{density}}\Big\{ \int\frac{2 a_N K}{nN^2} \sum_{x=0}^{nN-1} \big(\eta(x) - \eta(x+1)\big) f(\eta)\nu_\rho (d \eta) - \<-\gen \sqrt{f},\sqrt{f}\>_\rho\Big\}.\]
		Above, $f$ is a $\nu_\rho-$density if $f \geq 0$ and $\int f d \nu_\rho = 1$, and for two local functions $f,g: \Omega \rightarrow \R$,
		\[\<f,g\>_\rho = \int f(\eta) g(\eta)\,\nu_\rho (d \eta).\]
		We remark that although the initial distribution associated to the measure $\P$ is $\nu_\rho^*$, by the basic coupling stated in Remark \ref{rmk:initialmeasure}, there is no difference if we regard the initial measure as the equilibrium measure $\nu_\rho$. Therefore, in the following argument, we will not distinguish  processes with the above two initial measures. Then, direct calculations yield that
		\[ \<-\gen \sqrt{f},\sqrt{f}\>_\rho = \frac{1}{4} \sum_{x \in \Z} \int \big(\sqrt{f}(\eta^{x,x+1}) - \sqrt{f} (\eta)\big)^2 \nu_\rho (d \eta). \]
		Making the change of variables $\eta \mapsto \eta^{x,x+1}$ and using the Cauchy-Schwarz inequality, for any $A > 0$, we bound the first term in the above brace by
		\begin{multline*}
			\frac{a_N K}{nN^2} \sum_{x=0}^{nN-1} \int \big(\eta(x)-\eta(x+1)\big) \big( f (\eta) -f (\eta^{x,x+1}) \big) \nu_\rho (d \eta) 
			\\
			\leq  \frac{1}{2} \frac{a_N K}{nN^2} \sum_{x=0}^{nN-1} \Big\{  A \int \big(\sqrt{f}(\eta^{x,x+1}) - \sqrt{f} (\eta)\big)^2 \nu_\rho (d \eta) + \frac{1}{A}   \int \big(\sqrt{f}(\eta^{x,x+1}) + \sqrt{f} (\eta)\big)^2 \nu_\rho (d \eta) \Big\}\\
			\leq  \frac{a_N KA}{2 nN^2}  \<-\gen \sqrt{f},\sqrt{f}\>_\rho + \frac{2 a_NK}{NA}.
		\end{multline*}
		Taking $A = 2 nN^2 / (a_N K)$, the second term in \eqref{est_2} is bounded from above by
		\[\frac{tN^3}{4a_N^2} \times \frac{2 a_N K}{N} \times \frac{a_N K}{2 nN^2} = \frac{tK^2}{4n},\]
		which converges to zero as $n \rightarrow \infty$.
		
		In conclusion, we have shown that for any $K > 0$, 
		\[\limsup_{n \rightarrow \infty} 	\limsup_{N \rightarrow \infty} \frac{N}{a_N^2} \log \P \Big( \big| \frac{1}{nNa_N} \sum_{x=0}^{nN-1} J_{x,x+1} (tN^2)   \big|> \delta \Big) \leq - \delta K.\]
		We conclude the proof by letting $K \rightarrow \infty$.
	\end{proof}
	
	\subsection{Exponential tightness.} In this subsection, we show the rescaled current and tagged particle positions are exponentially tight.
	
	\begin{lemma}[Exponential tightness]\label{lem:exptight} We have the following estimates,
		\begin{align}
			\limsup_{M \rightarrow \infty} \limsup_{N \rightarrow \infty} \frac{N}{a_N^2} \log \bP \Big( \Big|  \frac{1}{a_N} J_{-1,0} (TN^2)\Big| > M \Big) = - \infty,\label{exptight_current}\\
			\limsup_{M \rightarrow \infty} \limsup_{N \rightarrow \infty} \frac{N}{a_N^2} \log \bP \Big( \Big|  \frac{1}{a_N} X (TN^2)\Big| > M \Big) = - \infty.\label{exptight_tp}
		\end{align}
	\end{lemma}
	
	\begin{proof}
		We first prove \eqref{exptight_current}. By Markov inequality, the expression in \eqref{exptight_current} is bounded by
		\[-M + \frac{N}{a_N^2} \log \bE \Big[ \exp \Big\{ \frac{a_N}{N} \big|   J_{-1,0} (TN^2)\big| \Big\}\Big].\]
		We shall prove the second term above is bounded uniformly for $N$ large enough. As in the proof of Lemma \ref{lem:superexponential}, we could remove the absolute value inside the exponential. Recall Eq. \eqref{current_empiricalM}. Then for any $n > 0$, we may rewrite the second term above as
		\[\frac{N}{a_N^2} \log \bE \Big[ \exp \Big\{ \frac{a_N^2}{N} \<\mu^N_T,G_n\> -  \frac{a_N^2}{N} \<\mu^N_0,G_n\> + \frac{a_N}{nN^2} \sum_{x=0}^{nN-1} J_{x,x+1} (TN^2) \Big\}\Big].\]
		By Cauchy-Schwarz inequality, the last formula is bounded by
		\begin{multline*}
			\frac{N}{3a_N^2} \log \bb{E} \Big[ \exp \Big\{ \frac{3a_N^2}{N} \<\mu^N_T,G_n\> \Big\} \Big] + 	\frac{N}{3a_N^2} \log \bb{E} \Big[ \exp \Big\{ -\frac{3a_N^2}{N} \<\mu^N_0,G_n\> \Big\} \Big] \\+ 	\frac{N}{3a_N^2} \log \bb{E} \Big[ \exp \Big\{ \frac{3a_N}{nN^2} \sum_{x=0}^{nN-1} J_{x,x+1} (TN^2) \Big\} \Big].
		\end{multline*}
		Taking $K=3$ in  the proof of Lemma \ref{lem:superexponential}, the third term above is bounded by $CT/n$ for some finite constant $C$. Next, we only deal with the first term above since the second one could be handled in the same way. Using the basic inequality $e^x \leq 1 + x + (x^2/2) e^{|x|}$, we could bound the the first term above by 
		\[\frac{C N}{a_N^2} \sum_{x \in \Z} \frac{a_N^2}{N^2} G_n^2 (\tfrac{x}{N}) e^{3a_N/N} \leq C n\]
		for $N$ large enough. This proves \eqref{exptight_current} by letting $M \rightarrow \infty$.
		
		\medspace
		
		Now we use \eqref{exptight_current} to prove \eqref{exptight_tp}. According to the spatial homogeneity of our process, 
		\[\P\Big(\Big|\frac{1}{a_N} X (TN^2)\Big|> M\Big)=2\P\Big(\frac{1}{a_N} X (TN^2)> M\Big).\]
		Therefore, we could remove the absolute value inside the probability  in \eqref{exptight_tp} and only need to consider the case  $J_{-1,0}(TN^2)>0$.
		Using the identity \eqref{jpositive},  for any $\theta>0$,
		\[
		\Big\{\frac{1}{a_N}X(TN^2)>M\Big\}\bigcap\Big\{\frac{1}{a_N}J_{-1,0}(TN^2)\leq \theta M\Big\}
		\subseteq \Big\{\frac{1}{a_N}\sum_{x=0}^{Ma_N}\eta_{TN^2}(x)\leq \theta M\Big\}
		\]
		Hence, for $0<\theta<\rho$,
		\begin{equation}\label{equ 4.7}
			\P\Big(\frac{1}{a_N}X(TN^2)>M\Big)\leq \P\Big(\frac{1}{a_N}J_{-1,0}(TN^2)\geq \theta M\Big)+\P\Big(\frac{1}{a_N}\sum_{x=0}^{Ma_N}\eta_{TN^2}(x)\leq \theta M\Big).  
		\end{equation}
		Since $\nu_\rho$ is an invariant distribution of the SSEP and $\theta<\rho$, according to the classic large deviation theory  of the sum of i.i.d. random variables,
		\[
		\limsup_{N\rightarrow+\infty}\frac{1}{a_N}\log\P\Big(\frac{1}{a_N}\sum_{x=0}^{Ma_N}\eta_{TN^2}(x)\leq \theta M\Big)<0.
		\]
		In particular,
		\begin{equation}\label{equ 4.8}
			\limsup_{N\rightarrow+\infty}\frac{N}{a_N^2}\log\P\Big(\frac{1}{a_N}\sum_{x=0}^{Ma_N}\eta_{TN^2}(x)\leq \theta M\Big)=-\infty.
		\end{equation}
		Equations \eqref{exptight_current}, \eqref{equ 4.7} and \eqref{equ 4.8} imply that
		\[
		\limsup_{M \rightarrow \infty} \limsup_{N \rightarrow \infty} \frac{N}{a_N^2} \log \bP \Big( \frac{1}{a_N} X (TN^2)> M \Big) = - \infty.
		\]
		This concludes the proof.
	\end{proof}

	\subsection{Weak MDP upper bound.}\label{subsec:upper}  In this subsection, we prove Theorem \ref{thm:main_1} for the upper bound. By the exponential tightness in Lemma \ref{lem:exptight}, we only need to prove the upper bound over closed intervals.

	We first prove the upper bound for the currents. For any closed interval $[a,b]$, $a < b$, let $[c_k,c_{k+1}], 1 \leq k \leq m$ be a partition of the interval. By \eqref{current_empiricalM},
	\begin{multline*}
		\limsup_{N \rightarrow \infty} \frac{N}{a_N^2} \log \P \Big(\frac{1}{a_N} J_{-1,0} (TN^2) \in [a,b]\Big) \\
		\leq \max_{1 \leq k \leq m} 	\limsup_{N \rightarrow \infty} \frac{N}{a_N^2} \log \P \Big(\<\mu^N_T,G_n\> - \<\mu^N_0,G_n\> + \frac{1}{nNa_N} \sum_{x=0}^{nN-1} J_{x,x+1} (TN^2) \in [c_k,c_{k+1}]\Big).
	\end{multline*}
	Using Lemma \ref{lem:superexponential}, the last line is bounded by
	\[	\limsup_{m \rightarrow \infty}	\limsup_{\delta \rightarrow 0} \limsup_{n \rightarrow \infty} \max_{1 \leq k \leq m} 	\limsup_{N \rightarrow \infty} \frac{N}{a_N^2} \log \P \Big(\<\mu^N_T,G_n\> - \<\mu^N_0,G_n\>  \in [c_k-\delta,c_{k+1}+\delta]\Big).\]
	By Theorem \ref{thm:mdphl}, we bound the last formula by
	\begin{equation}\label{current_upper}
		\limsup_{m \rightarrow \infty}	\limsup_{\delta \rightarrow 0} \limsup_{n \rightarrow \infty} \max_{1 \leq k \leq m}  - \inf \{\mathcal{Q} (\mu):  \<\mu_T,G_n\> - \<\mu_0,G_n\>  \in [c_k-\delta,c_{k+1}+\delta]\}.
	\end{equation}
	For any $\varepsilon > 0$, we could find a $\mu^{k,n,\delta}_\varepsilon$ such that
	\[\<\mu_\varepsilon^{k,n,\delta} (T,\cdot),G_n\> - \<\mu_\varepsilon^{k,n,\delta} (0,\cdot),G_n\>  \in [c_k-\delta,c_{k+1}+\delta]\]
	and that the infimum in \eqref{current_upper} is bounded from below by $\mathcal{Q} (\mu_\varepsilon^{k,n,\delta}) - \varepsilon$. We claim that there exists a constant $C_0 $ such that for any $\delta,k$ and $n$ large enough, 
	\begin{equation}\label{q_upper}
		\mathcal{Q} (\mu_\varepsilon^{k,n,\delta}) \leq C_0 + \varepsilon.
	\end{equation}
	This permits us to extract a subsequence on which the limsup of $- \mathcal{Q} (\mu_\varepsilon^{k,n,\delta})$ is attained as $n \uparrow \infty$, $\delta \downarrow 0$ and $m \uparrow \infty$ and on which also $\mu_\varepsilon^{k,n,\delta} \rightarrow \mu_\varepsilon^*$ in $D([0,T],\mathcal{S}^\prime (\R))$ for some $\mu_\varepsilon^*$. We could also choose the subsequence such that $c_k \rightarrow \alpha \in [a,b]$ for some $\alpha$ since $[a,b]$ is compact. Moreover, 
	\[\<\mu_\varepsilon^* (T, \cdot), \chi_{[0,\infty)}\>- \<\mu_\varepsilon^* (T, \cdot), \chi_{[0,\infty)}\> = \alpha.\]
	Therefore,  \eqref{current_upper} is bounded from above by
	\[\sup_{\alpha \in [a,b]} - \{\mathcal{Q}(\mu^*_\varepsilon),\,\<\mu_\varepsilon^* (T, \cdot), \chi_{[0,\infty)}\>- \<\mu_\varepsilon^* (T, \cdot), \chi_{[0,\infty)}\> = \alpha\} + \varepsilon \leq  - \inf_{\alpha \in [a,b]} \bb{J} (\alpha) + \varepsilon.\]
	Letting $\varepsilon \rightarrow 0$, we conclude the proof for the upper bound for the currents over closed intervals.
	
	It remains to prove $\eqref{q_upper}$. Recall $\tilde{\mu}_{\alpha,T}$ defined in  \eqref{mu_tilde} and that
	\[\int_0^T \tilde{J}_{\alpha,T} (t,0) dt = \int_0^\infty \tilde{\mu}_{\alpha,T} (T,u) - \tilde{\mu}_{\alpha,T} (0,u) du = \alpha.\]
	Then 
	\begin{multline}\label{upper1}
		\<\tilde{\mu}_{\alpha,T}(T,\cdot),G_n\> - \<\tilde{\mu}_{\alpha,T} (0,\cdot),G_n\>  = \int_{0}^T \int_{\R} \partial_t \tilde{\mu}_{\alpha,T} (t,u) G_n (u) du dt \\
		= - \int_{0}^T \int_{\R} \partial_u \tilde{J}_{\alpha,T} (t,u) G_n (u) du dt 
		= \int_0^T \tilde{J}_{\alpha,T} (t,0) dt - \frac{1}{n} \int_0^T \int_0^n \tilde{J}_{\alpha,T} (t,u) du dt\\
		= \alpha + \frac{1}{n} \int_0^T \int_0^n \frac{1}{2} \partial_u \tilde{\mu}_{\alpha,T} (t,u) - \rho (1-\rho) \partial_u \tilde{H}_{\alpha,T} (t,u)\,du\,dt.
	\end{multline}
	Since  $\partial_u \tilde{\mu}_{\alpha,T},\, \partial_u \tilde{H}_{\alpha,T} \in L^2 ([0,T] \times \R)$, by Cauchy-Schwarz inequality, the second term on the right-hand side is bounded by $Cn^{-1/2}$ for some finite constant $C$. Therefore, for any $\alpha \in [a,b]$ and  for any $n$ large enough,
	\[	\<\tilde{\mu}_{\alpha,T}(T,\cdot),G_n\> - \<\tilde{\mu}_{\alpha,T}(0,\cdot),G_n\> \in [\alpha-\delta,\alpha + \delta].\]
	Therefore, the infimum in \eqref{current_upper} is bounded from above by
	\[\sup_{\alpha \in [a,b]} \mathcal{Q} (\mu_\alpha) = \sup_{\alpha \in [a,b]} \mathbb{J} (\alpha) =: C_0.\]
	This proves the claim \eqref{q_upper}.
	
	\medspace
	
	Now we prove the upper bound for the tagged particle positions. We only  prove the upper bound for closed positive intervals $[a, b]$, $0<a<b<+\infty$, and the other cases could be handled in the same way. For sufficiently small $\delta$ and sufficiently large $N$, let $D_{N, a,b,\delta}$ be the event that
	\[
	\sup_{j\in [aa_N-1, ba_N]}\frac{1}{a_N}\Big|\sum_{x=0}^{j}\big(\eta_{TN^2}(x)-\rho\big)\Big|<\delta.
	\]
	Then, the identity  \eqref{jpositive} implies that
	\[
	\Big\{\frac{1}{a_N}X(TN^2)\in[a,b]\Big\}\bigcap D_{N,a,b,\delta}\subseteq \Big\{\frac{1}{a_N}J_{-1,0}(TN^2)\geq \rho a-\delta\Big\}.
	\]
	By standard large deviation theory, for any  $0\leq c <+\infty$ and $\delta>0$, 
	\[
	\limsup_{N\rightarrow+\infty}\frac{1}{a_N}\log \P\Big(\sup_{j\in[0, c a_N]}\Big|\frac{1}{a_N}\sum_{x=0}^j(\eta_{TN^2}(x)-\rho)\Big|>\delta\Big)<0. 
	\]
	In particular,
	\[
	\limsup_{N\rightarrow+\infty}\frac{N}{a_N^2}\log \P\Big( D_{N,a,b,\delta}^c \Big)=-\infty. 
	\]
	Then, according to Theorem \ref{thm:main_2} and  moderate deviation upper bounds for the current,
	\begin{multline*}
		\limsup_{N\rightarrow+\infty}\frac{N}{a_N^2}\log \P\Big(\frac{1}{a_N}X(TN^2)\in [a,b]\Big) 
		=\limsup_{N\rightarrow+\infty}\frac{N}{a_N^2}\log \P\Big(\frac{1}{a_N} X(TN^2)\in [a,b], D_{N,a,b,\delta}\Big)\\
		\leq \limsup_{N\rightarrow+\infty}\frac{N}{a_N^2}\log \P\Big(\frac{1}{a_N} J_{-1,0}(TN^2)\geq \rho a-\delta\Big)
		\leq-\inf_{u\geq \rho a-\delta}\bb{J}(u)
		=-\frac{\sqrt{2\pi}\big(\rho a-\delta\big)^2}{4\rho(1-\rho)\sqrt{T}}.
	\end{multline*}
	Since $\delta$ is arbitrary, let $\delta\rightarrow 0$, then by Theorem \ref{thm:main_2}
	\[\limsup_{N\rightarrow+\infty}\frac{N}{a_N^2}\log \P\Big(\frac{1}{a_N} X(TN^2)\in [a,b]\Big) \leq -\frac{\sqrt{2\pi}\rho a^2}{4(1-\rho)\sqrt{T}}
	=-\inf_{u\in [a, b]}\bb{I}(u).\]
	This concludes the proof of the weak upper bound for the tagged particle positions.
	
	\subsection{Lower bound.}\label{subsec:lower} In this section, we prove Theorem \ref{thm:main_1} for the lower bound. As before, we first investigate the current. Let $O \subset \R$ be a nonempty open set. Let $\alpha \in O$ and $\varepsilon > 0$ such that $(\alpha-\varepsilon,\alpha+\varepsilon) \subset O$. By \eqref{current_empiricalM},  
	\begin{multline*}
		\liminf_{N \rightarrow \infty} \frac{N}{a_N^2} \log \P \Big(\frac{1}{a_N} J_{-1,0} (TN^2) \in O\Big) \\
		\geq 	\limsup_{N \rightarrow \infty} \frac{N}{a_N^2} \log \P \Big(\<\mu^N_T,G_n\> - \<\mu^N_0,G_n\> + \frac{1}{nNa_N} \sum_{x=0}^{nN-1} J_{x,x+1} (TN^2) \in (\alpha-\varepsilon,\alpha+\varepsilon)\Big).
	\end{multline*}
	By Lemma \ref{lem:superexponential}, taking $\delta = \varepsilon /2$, we bound the last line from below by
	\[ \liminf_{n \rightarrow \infty} 	\liminf_{N \rightarrow \infty} \frac{N}{a_N^2} \log \P \Big(\<\mu^N_T,G_n\> - \<\mu^N_0,G_n\>  \in [\alpha-\varepsilon/2,\alpha+\varepsilon/2]\Big).\]
	By Theorem \ref{thm:mdphl}, the last formula is bounded by
	\[ \liminf_{n \rightarrow \infty}  - \inf \{\mathcal{Q} (\mu):  \<\mu_T,G_n\> - \<\mu_0,G_n\>  \in [\alpha-\varepsilon /2,\alpha + \varepsilon/2] \}\]
	Taking $\mu = \tilde{\mu}_{\alpha,T}$ defined in  \eqref{mu_tilde} and by \eqref{upper1}, the last line is bounded below by $- \mathcal{Q} (\tilde{\mu}_{\alpha,T}) = - \mathbb{J} (\alpha)$. Taking the supremum over $\alpha \in O$, we have
	\[\liminf_{N \rightarrow \infty} \frac{N}{a_N^2} \log \P \Big(\frac{1}{a_N} J_{-1,0} (TN^2) \in O\Big) \geq - \inf_{\alpha \in O}  \mathbb{J} (\alpha).\]
	
	\medspace
	
	Now we prove  the lower bound of the tagged particle positions. By Theorem \ref{thm:main_2},  the rate function $\bb{I}$ is even and increasing in $[0, +\infty)$. Hence we only need to check the lower bound for open set $O$ with forms $(a,+\infty)$ or $(-\infty, -a)$ for any $a>0$. According to the spatial homogeneity of the process, we only need to deal with the first case. For $a, \delta>0$, let $B_{N,a,\delta}$ be the event that
	\[
	\sup_{j\in [0, aa_N]}\frac{1}{a_N}\Big|\sum_{x=0}^{j}\big(\eta_{TN^2}(x)-\rho\big)\Big|<\delta,
	\]
	then we have shown in the last subsection that
	\[
	\limsup_{N\rightarrow+\infty}\frac{N}{a_N^2}\log \P\Big( B_{N,a,\delta}^c \Big)=-\infty. 
	\]
	Moreover, \eqref{jpositive} implies that
	\[
	\Big\{\frac{1}{a_N}J_{-1,0}(TN^2)>\rho a+2\delta\Big\}\bigcap B_{N,a,\delta}\subseteq \Big\{\frac{1}{a_N}X(TN^2)>a\Big\}.
	\]
	According to Theorem \ref{thm:main_2} and moderate deviation lower bounds of the current, 
	\begin{align*}
		\liminf_{N\rightarrow+\infty}\frac{N}{a^2_N}\log \P\Big(\frac{1}{a_N}X(TN^2)>a\Big)
		&\geq \liminf_{N\rightarrow\infty}\frac{N}{a^2_N}\log \P\Big(\frac{1}{a_N}J_{-1,0}(TN^2)>\rho a+2\delta, B_{N, a, \delta}\Big)\\
		&=\liminf_{N\rightarrow\infty}\frac{N}{a^2_N}\log \P\Big(\frac{1}{a_N}J_{-1,0}(TN^2)>\rho a+2\delta\Big)\\
		&\geq -\inf_{u>\rho a+2\delta}\bb{J}(u)=-\frac{\sqrt{2\pi}\big(\rho a+2\delta\big)^2}{4\rho(1-\rho)\sqrt{T}}.
	\end{align*}
	Since $\delta$ is arbitrary, let $\delta\rightarrow 0$, then by Theorem \ref{thm:main_2}
	\begin{align*}
		\liminf_{N\rightarrow+\infty}\frac{N}{a^2_N}\log \P\Big(\frac{1}{a_N}X(TN^2)>a\Big)
		\geq -\frac{\sqrt{2\pi}\rho a^2}{4(1-\rho)\sqrt{T}}=-\inf_{u>a}\bb{I}(u).
	\end{align*}
	This concludes the proof of the lower bound.

	\bibliographystyle{plain}
	\bibliography{tagged}
\end{document}